\documentclass[11pt]{article} 
\usepackage{geometry} 
\usepackage[utf8]{inputenc} 
\usepackage[T1]{fontenc} 
\usepackage[english]{babel} 
\usepackage{amsmath,amsthm} 
\usepackage{amssymb,yfonts}
\usepackage{nicefrac} 
\usepackage{hyperref}
\usepackage{graphicx}
\usepackage{braket}
\usepackage{subfig}
\usepackage{color}
\usepackage[usenames,dvipsnames]{xcolor}
\usepackage{csquotes}
\usepackage{dsfont}
\usepackage[sc]{mathpazo}
\usepackage{comment}

\newtheorem{theorem}{Theorem}[section]

\newtheorem{remark}[theorem]{Remark}
\newtheorem{lemma}[theorem]{Lemma}
\newtheorem{example}[theorem]{Example}
\newtheorem{proposition}[theorem]{Proposition}
\numberwithin{equation}{section}
\theoremstyle{definition}

\newcommand{\beq}{\begin{equation}}
\newcommand{\eeq}{\end{equation}}

\newcommand{\eps }{\varepsilon}

\newcommand{\R}{\mathbb{R}}

\def\r{\mathbb{R}}
\def\rn{\mathbb{R}^N}
\def\z{\mathbb{Z}}

\def\n{\mathbb{N}}
\def\cc{\mathbb{C}}
\def\eps{\varepsilon}
\def\rh{\rightharpoonup}

\def\irn{\displaystyle\int_{\r^N}}
\def\vp{\varphi}

\def\cC{\mathcal{C}}

\def\cN{\mathcal{N}}

\def\bar{\overline}
\def\hat{\widehat}

\def\d{\,\mathrm{d}}

\title{An upper bound for the least energy of a sign-changing solution to a zero mass problem}
\date{\today}
\author{M\'onica Clapp\footnote{M. Clapp was supported by CONACYT (Mexico) through the research grant A1-S-10457.}, \quad Liliane A. Maia\footnote{L. Maia was supported by FAPDF, CNPq/PQ 309866/2020-0 (Brazil), and PROEX/CAPES (Brazil).} \quad and \quad Benedetta Pellacci\footnote{B. Pellacci was supported  by the PRIN-2017-JPCAPN Grant: “Equazioni differenziali alle derivate parziali non lineari”, by the project Vain-Hopes within the program VALERE: VAnviteLli pEr la RicErca and by the INdAM-GNAMPA group.}}

\begin{document}

\maketitle

\begin{abstract}
We give an upper bound for the least energy of a sign-changing solution to the the nonlinear scalar field equation
$$-\Delta u = f(u), \qquad u\in D^{1,2}(\mathbb{R}^{N}),$$
where $N\geq5$ and the nonlinearity $f$ is subcritical at infinity and supercritical near the origin. More precisely, we establish the existence of a  nonradial sign-changing solution whose energy is smaller that $12c_0$ if $N=5,6$ and smaller than $10c_0$ if $N\geq 7$, where $c_0$ is the ground state energy.
\smallskip

\smallskip

\end{abstract}

\section{Introduction}
\label{sec:introduction}

The aim of this note is to give an upper bound for the least energy of a sign-changing solution to the problem 
\begin{equation}
\label{prob}
-\Delta u = f(u), \qquad u\in D^{1,2}(\R^{N}),
\end{equation}
where $N\geq5$ and the nonlinearity $f$ is subcritical at infinity and supercritical near the origin. More precisely, we assume
\begin{enumerate}
\item[$(f_1)$] $f\in\cC_\mathrm{loc}^{1,\alpha}(\r)$ with $\alpha\in\big(\frac{N}{2(N-2)},1\big]$, and there exist $a_{1}>0$ and $2<p<2^*:=\frac{2N}{N-2}<q$ such that, for $\kappa=-1,0,1$,
\begin{equation}
|f^{(\kappa)}(s)|\leq
\begin{cases}
a_{1}|s|^{p-(\kappa+1)} & \text{if}\ |s|\geq 1,\\
a_{1}|s|^{q-(\kappa+1)} & \text{if}\ |s|\leq 1,
\end{cases}
\end{equation}
where $f^{(-1)}:=F,$ $f^{(0)}:=f,$ $f^{(1)}:=f',$ and $F(s):=\int_{0}^{s}f(t)\mathrm{d}t.$
\item[$(f_2)$] There is a constant $\theta>2$ such that $0\leq\theta F(s)\leq f(s)s<f'(s)s^{2}$ for all $s>0$.
\item[$(f_3)$] $f$ is odd.
\end{enumerate}

In their seminal paper \cite{bl} Berestycki and Lions showed that problem \eqref{prob} has a ground state solution which is positive, radially symmetric and decreasing in the radial direction. One or multiple positive solutions for a similar equation involving a scalar potential that decays to zero at infinity, both in the whole space and in an exterior domain, have been obtained, for instance, in \cite{bm,cm,cmp,km}.

The existence of nonradial sign-changing solutions to \eqref{prob} was recently shown by Mederski in \cite{m}. It is readily seen that the energy of any sign-changing solution must be greater than twice the energy of the ground state. But, to our knowledge, there are no upper estimates for the least energy of a sign-changing solution to \eqref{prob}.

Our aim is to prove the following result.

\begin{theorem}\label{thm:main}
Assume that  $f$ satisfies $(f_1)-(f_3)$. Then, there exists a nonradial sign-changing solution $\hat\omega$ to the problem \eqref{prob} whose energy satisfies
\begin{equation}
2c_0<\frac12\irn|\nabla \hat\omega|^{2}-\irn F(\hat\omega)<
\begin{cases}
12c_0 &\text{if \ }N=5,6,\\
10c_0 &\text{if \ }N\geq 7,
\end{cases}
\end{equation}
where $c_0$ is the ground state energy of \eqref{prob}. Furthermore, for each $(z_{1},z_{2},y)\in \R^{N}\equiv\mathbb{C} \times\mathbb{C} \times \mathbb{R}^{N-4}$,
\begin{enumerate}
\item[$(a)$]
$ \hat\omega(z_{1},z_{2},y)=\hat\omega(\mathrm{e}^{2\pi\mathrm{i}j/5}z_{1},\mathrm{e}^{2\pi\mathrm{i}j/5}z_{2},y) $ \ for all $j=0,\ldots,4$,
\item[$(b)$] $\hat\omega(z_{1},z_{2},y) = -\hat\omega(z_{2},z_{1},y)$,  
\item[$(c)$] $\hat\omega(z_{1},z_{2},y_1) = \hat\omega(z_{1},z_{2},y_2)$ \ if $|y_1| = |y_2|$,
\end{enumerate}
and $\hat\omega$ has least energy among all nontrivial solutions satisfying $(a), (b), (c)$.
\end{theorem}

In the positive mass case, for the subcritical pure power nonlinearity $f(u)=|u|^{p-2}u$ with $2<p<2^*$, an estimate for the least energy of a sign-changing solution was obtained in \cite[Theorem 1.1]{cs}, and recently improved in \cite[Corollary 1.2]{cso}. On the other hand, it was shown in \cite[Theorem 1.1]{cpw} that the same estimates as in Theorem \ref{thm:main} hold true for the critical pure power nonlinearity $f(u)=|u|^{2^*-2}u$.

As in \cite{m}, to prove the existence of a sign-changing solution to \eqref{prob} we take advantage of suitable symmetries that produce a change of sign by construction. The symmetries introduced in \cite{m}, however, have only infinite and trivial orbits. This does not allow estimating the energy of the solution. Here, in contrast, we consider symmetries given by a finite group. This makes it harder to show existence due to the lack of compactness but, once the existence of a solution is established, one immediately gets an upper estimate for its energy. 

As in \cite{cs,cso,cpw} we use concentration compactness techniques to establish a condition for the existence of a symmetric minimizer for the variational problem associated to \eqref{prob}. Then, we consider a suitable ansatz given as a sum of positive and negative copies of the ground state solution placed along some orbit in a convenient way. But, unlike in the subcritical case where the decay of the ground state is exponential, here the decay is polynomial. A careful estimate for the interaction among the terms of the ansatz, that applies to this situation, is provided by Lemma \ref{lem:cm}.

Another delicate issue is produced by the fact that our nonlinearity is not the pure power one. This asks for a proper estimate of its lack of additivity. Lemma \ref{lem:acp2} gives such an estimate for a nonlinearity $f$ satisfying milder regularity assumptions than those considered in \cite{acp}.

The condition $\alpha\in\big(\frac{N}{2(N-2)},1\big]$ in assumption $(f_1)$ implies that $N\geq 5$. This condition is needed to ensure that some terms measuring the deviation from additivity of the nonlinearity are lower order terms, as pointed out in the proof of Proposition \ref{prop:upperbound}.

This paper is organized as follows. In Section \ref{sec:setting} we study the symmetric variational problem and give a condition for the existence of a symmetric minimizer. In Section \ref{sec:upperbound} we show that this condition is satisfied for some particular symmetries and we prove Theorem \ref{thm:main}. In the appendices we prove some lemmas required to achieve that purpose.

\section{The symmetric variational setting}
\label{sec:setting}

We assume throughout that $f$ satisfies $(f_1)$, $(f_2)$ and $(f_3)$.

Let $G$ be a closed subgroup of the group $O(N)$ of linear isometries of $\rn$ and denote by
$$Gx:=\{gx:g\in G\}\qquad\text{and}\qquad G_x:=\{g\in G:gx=x\}$$ 
the $G$-orbit and the $G$-isotropy group of a point $x\in\rn$. \ The $G$-orbit $Gx$ is $G$-homeomorphic to the homogeneous space $G/G_x$. So both have the same cardinality, i.e., $|Gx|=|G/G_x|$.

Let $\phi:G\to\z_2:=\{-1,1\}$ be a continuous homomorphism of groups satisfying
\begin{itemize}
\item[$(A_\phi)$] If $\phi$ is surjective, then there exists $\zeta\in\rn$ such that $(\ker\phi)\zeta\neq G\zeta$,
\end{itemize} 
where $\ker\phi:=\{g\in G:\phi(g)=1\}$. A function $u:\rn\to\r$ such that
$$u(gx)=\phi(g)u(x) \text{ \ for all \ }g\in G, \ x\in\rn,$$
will be called \emph{$\phi$-equivariant}. If $\phi\equiv 1$ is the trivial homomorphism then $u$ is $G$-invariant, i.e., it is constant on every $G$-orbit, while if $\phi$ is surjective and $u\neq 0$ then $u$ is nonradial and changes sign. If $K$ is a closed subgroup of $G$ we write $\phi|K:K\to\z_2$ for the restriction of $\phi$ to $K$. Note that $\phi|K$ satisfies $(A_{\phi|K})$ if $\phi$ satisfies $(A_\phi)$.

Let $D^{1,2}(\mathbb{R}^{N}):=\{u\in L^{2^*}(\mathbb{R}^{N}):\nabla u\in L^{2}(\mathbb{R}^{N},\mathbb{R}^{N})\}$, with its standard scalar product and norm
\[
\langle u,v\rangle :=\int_{\mathbb{R}^{N}}\nabla u\cdot\nabla v,\text{\qquad}\|u\| :=\left(  \int_{\mathbb{R}^{N}}|\nabla u|^{2}\right)^{1/2},
\]
and set 
$$D^{1,2}(\rn)^\phi:=\{u\in D^{1,2}(\rn): u\text{ is }\phi\text{-equivariant}\}.$$
Assumption $(A_\phi)$ guarantees that $D^{1,2}(\rn)^\phi$ has infinite dimension, see \cite{bcm}. 

By the principle of symmetric criticality \cite{p}, the $\phi$-equivariant solutions to the problem $\eqref{prob}$ are the critical points of the functional \ $J:D^{1,2}(\rn)^\phi\to\r$ \ given by
\[
J(u):=\frac12\|u\|^{2}-\irn F(u),
\]
where $F(u):=\int_{0}^{u}f(s)\,\d s$. This functional is well defined and of class  $\mathcal{C}^{2},$ with derivative
\[
J^{\prime}(u)v=\int_{\mathbb{R}^{N}}\nabla u\cdot\nabla v -\int_{\mathbb{R}^{N}}f(u)v,\qquad u,v\in D^{1,2}(\mathbb{R}^{N})^\phi;
\]
see \cite[Proposition 3.8]{bpr} or \cite[Lemma 2.6]{bm}. The nontrivial $\phi$-equivariant solutions belong to the set
$$
\mathcal{N}^\phi:=\{u\in D^{1,2}(\R^{N})^\phi:u\neq 0,\,J'(u)u=0\},
$$
which is a closed $\mathcal{C}^{1}$-submanifold of $D^{1,2}(\R^{N})^\phi$ and a natural constraint for $J$, and
$$c^\phi:=\inf_{u\in\cN^\phi}J(u)>0;$$
see \cite[Lemma 3.2]{cm}.

Next, we give a description of the minimizing sequences for $J$ on $\cN^\phi$. We need the following lemmas. Set $B_{R}(y):=\{x\in\mathbb{R}^{N}:\left\vert x-y\right\vert <R\}.$

\begin{lemma} \label{lem:lions}
If $(u_{k})$ is bounded in $D^{1,2}(\mathbb{R}^{N})$ and there exists $R>0$ such that
\[
\lim_{k\to\infty}\left(\sup_{y\in\rn}\int_{B_R(y)}|u_k|^2\right)=0,
\]
then \ $\lim_{k\to\infty}\irn f(u_k)u_k=0$.
\end{lemma}

\begin{proof}
See \cite[Lemma 3.5]{cm}.
\end{proof}

\begin{lemma} \label{lem:bl}
If $u_{k}\rh u$ weakly in $D^{1,2}(\rn)$ then, passing to a subsequence,
\begin{enumerate}
\item[$(a)$] $\irn|f(u_k)-f(u)||\vp| = o(1)$ \ for every $\vp\in\cC_c^{\infty}(\rn)$,
\item[$(b)$] $ \irn F(u_{k}) - \irn F(u_{k}-u) = \irn F(u)+o(1)$,
\item[$(c)$]$\irn f(u_k)u_k - \irn f(u_k - u)[u_k- u] = \irn f(u)u + o(1),$
\item[$(d)$] $f(u_{k})-f(u_{k}-u)\to f(u)$ in $(D^{1,2}(\rn))'$.
\end{enumerate}
\end{lemma}

\begin{proof}
See \cite[Lemma 3.8]{cm}.
\end{proof}

\begin{lemma}\label{lem:orbits} 
For any given sequence $(y_k)$ in $\rn$, there exist a sequence $(\zeta_k)$ in $\rn$ and a closed subgroup $K$ of $G$ such that, up to a subsequence, the following statements hold true:
\begin{itemize}
\item[$(a)$] $\mathrm{dist}(Gy_k,\zeta_k)\leq C_0$ for all $k\in\n$ and some positive constant $C_0$.
\item[$(b)$] $G_{\zeta_k} = K$ for all $k\in\n$.
\item[$(c)$] If $|G/K| <\infty$ then $\lim\limits_{k\to\infty}|g\zeta_k-\bar g\zeta_k|=\infty$ for any pair $g,\bar g\in G$ with $\bar g g^{-1}\notin K$.
\item[$(d)$] If $|G/K|=\infty$ then, for each $n\in\n$, there exists $g_n\in G$ such that $\lim\limits_{k\to\infty}|g_{m}\zeta_k-g_{n}\zeta_k|=\infty$ for any $m,n\in\n$ with $m\neq n$.
\end{itemize}
\end{lemma}
        
\begin{proof}
These statements follow from \cite[Lemma 3.2]{ccs}.
\end{proof}

The proof of the next statement follows \cite[Theorem 2.1]{cs}. We give the details for the sake of completeness.

\begin{proposition} \label{prop:compactness}
If 
$$c^{\phi} < |G/G_\xi|\,c^{\phi|G_\xi}\qquad\text{for every \ }\xi\in\rn \ \text{with} \ G_\xi\neq G,$$
then $c^{\phi}$ is attained by $J$ on $\cN^\phi$.
\end{proposition}

\begin{proof}
Let $u_k\in\cN^{\phi}$ be such that $J(u_k)\to c^{\phi}$. Then,
$$\irn f(u_k)u_k=\|u_k\|^2\geq 2J(u_k)\geq c^\phi>0$$
for large enough $k$. Thus, after passing to a subsequence, Lemma \ref{lem:lions} yields $\delta>0$ and $y_k\in\rn$ such that
\begin{equation} \label{eq:nontrivial}
\int_{B_1(y_k)}|u_k|^2=\sup_{x\in\rn}\int_{B_1(x)}|u_k|^2\geq\delta\qquad\text{for all \ }k\in\n.
\end{equation}
For the sequence $(y_k)$ we choose a sequence $(\xi_{k})$ in $\rn$, a closed subgroup $K$ of $G$ and $C_0>0$ as in Lemma \ref{lem:orbits}, and we set $w_k(x):=u_k(x+\xi_k)$. Then, $w_k\in\cN^{\phi|K}$ and $J(w_k)\to c^{\phi}$. From assumption $(f_2)$ we get that $(w_k)$ is bounded in $D^{1,2}(\rn)$; see \cite[Lemma 3.6]{cm}. So, after passing to a subsequence, $w_k\rh w$ weakly in $D^{1,2}(\rn)$, $w_k\to w$ a.e. in $\rn$ and $w_k\to w$ strongly in $L^2_\mathrm{loc}(\rn)$. As $B_1(g_ky_k)\subset B_{C_0+1}(\xi_k)$ for some $g_k\in G$ and $|u_k|$ is $G$-invariant, \eqref{eq:nontrivial} yields
\begin{equation*}
\int_{B_{C_0+1}(0)}|w_k|^2=\int_{B_{C_0+1}(\xi_k)}|u_k|^2\geq\int_{B_1(g_ky_k)}|u_k|^2=\int_{B_1(y_k)}|u_k|^2\geq\delta\qquad\text{for all \ }k\in\n.
\end{equation*}
Therefore $w\neq 0$ and, using Ekeland's variational principle and Lemma \ref{lem:bl}, we see that $w$ solves \eqref{prob}. Note that, as each $w_k$ is $\phi|K$-equivariant, $w$ is also $\phi|K$-equivariant.

Let $g_1,\ldots,g_n\in G$ be such that $|g_j\xi_{k}-g_i\xi_{k}|\to\infty$ if $j\neq i$. Then, 
for each $j\in \{1,\dots,n\}$
$$\phi(g_j)\,(w_k\circ g_j^{-1})-\sum_{i=j+1}^n\phi(g_i)\,(w\circ g_i^{-1})(\,\cdot\,-g_i\xi_k+g_j\xi_k)\rh\phi(g_j)\,(w\circ g_j^{-1})$$
weakly in $D^{1,2}(\rn)$, where the sum is defined to be zero if $j=n$. It follows that
\begin{align*}
&\Big\|\phi(g_j)\,(w_k\circ g_j^{-1})-\sum_{i=j+1}^n\phi(g_i)\,(w\circ g_i^{-1})(\,\cdot\,-g_i\xi_k+g_j\xi_k)\Big\|^2\\
&=\Big\|\phi(g_j)\,(w_k\circ g_j^{-1})-\sum_{i=j}^n\phi(g_i)\,(w\circ g_i^{-1})(\,\cdot\,-g_i\xi_k+g_j\xi_k)\Big\|^2+\Big\|\phi(g_j)\,(w\circ g^{-1}_j)\Big\|^2+o(1).
\end{align*}
Since $u_k$ is $\phi$-equivariant, performing the change of variable $y=z-g_j\xi_k$ and recalling that $w_{k}(x)=u_{k}(x+\zeta_{k})$, we obtain
\begin{align*}
&\Big\|u_k-\sum_{i=j+1}^n\phi(g_i)\,(w\circ g_i^{-1})(\,\cdot\,-g_i\xi_k)\Big\|^2 \\
&=\Big\|u_k-\sum_{i=j}^n\phi(g_i)\,(w\circ g_i^{-1})(\,\cdot\,-g_i\xi_k)\Big\|^2+\|w\|^2+o(1),
\end{align*}
and iterating this identity for $j=1,\dots,n$, we deduce that
\begin{equation*}
\|u_k\|^2-\Big\|u_k-\sum_{i=1}^n\phi(g_i)\,(w\circ g_i^{-1})(\,\cdot\,-g_i\xi_k)\Big\|^2+n\|w\|^2+o(1).
\end{equation*}
Similarly, using Lemma \ref{lem:bl} we get
\begin{align*}
\irn F(u_k) &= \irn F\Big(u_k-\sum_{i=1}^n\phi(g_i)\,(w\circ g_i^{-1})(\,\cdot\,-g_i\xi_k)\Big)+n\irn F(w)+o(1), \\
\irn f(u_k)u_k &= \irn f\Big(u_k-\sum_{i=1}^n\phi(g_i)\,(w\circ g_i^{-1})(\,\cdot\,-g_i\xi_k)\Big)\Big[u_k-\sum_{i=1}^n\phi(g_i)\,(w\circ g_i^{-1})(\,\cdot\,-g_i\xi_k)\Big] \\
&\qquad+n\irn f(w)w+o(1).
\end{align*}
As $u_k,w\in\cN^{\phi|K}$ we derive
\begin{align*}
&\Big\|u_k-\sum_{i=1}^n\phi(g_i)\,(w\circ g_i^{-1})(\,\cdot\,-g_i\xi_k)\Big\|^2 \\
&=\irn f\Big(u_k-\sum_{i=1}^n\phi(g_i)\,(w\circ g_i^{-1})(\,\cdot\,-g_i\xi_k)\Big)\Big[u_k-\sum_{i=1}^n\phi(g_i)\,(w\circ g_i^{-1})(\,\cdot\,-g_i\xi_k)\Big],
\end{align*}
and using $(f_2)$ we obtain
$$c^\phi\geq\lim_{k\to\infty}J(u_k)\geq nJ(w)\geq nc^{\phi|K}.$$
Then, $|G/K|<\infty$ by Lemma \ref{lem:orbits}. If $K\neq G$, taking $n =|G/K|$ we get that $c^\phi\geq |G/K|c^{\phi|K}$, contradicting our assumption. Therefore, $K=G$, so $w\in\cN^\phi$ and $J(w)=c^\phi$, as claimed.
\end{proof}

\section{An upper bound for the energy of symmetric minimizers}
\label{sec:upperbound}

In \cite[Theorem 4]{bl} Berestycki and Lions established the existence of a ground state solution $\omega\in\mathcal{C}^{2}(\mathbb{R}^{N})$ to \eqref{prob}, which is positive, radially symmetric and decreasing in the radial direction. It satisfies the decay estimates
\begin{equation}\label{eq:decay}
\begin{split}
0<b_1(1+|x|)^{-(N-2)}\leq \omega(x) \leq b_2(1+|x|)^{-(N-2)},  
\\
|\nabla \omega(x)|\leq b_3(1+|x|)^{-(N-1)},  
\end{split}
\end{equation}
for all $x\in\rn$; see \cite[Theorem 1.1 and Corollary 1.2]{v}. 

From now on we consider the following symmetries.

\begin{example} \label{example}
We write $\rn\equiv\cc\times\cc\times\r^{N-4}$ and a point in $\rn$ as $(z_1,z_2,y)\in\cc\times\cc\times\r^{N-4}$. For $m\in\n$, let 
\[\z_m:=\{\mathrm{e}^{2\pi\mathrm{i}j/m}:j=0,\ldots,m-1\},\]
$G_m$ be the group generated by $\z_m\cup\{\tau\}$, acting on $\rn$ as 
\begin{align*}
\mathrm{e}^{2\pi\mathrm{i}j/m}(z_1,z_2,y):=(\mathrm{e}^{2\pi\mathrm{i}j/m}z_1,\mathrm{e}^{2\pi\mathrm{i}j/m}z_2,y), \qquad\qquad\tau(z_{1},z_{2},y):=(z_{2},z_{1},y),
\end{align*}
and $\phi:G_m\to\z_2$ be the homomorphism satisfying $\phi(\mathrm{e}^{2\pi\mathrm{i}j/m})=1$ and $\phi(\tau)=-1$. Set $\zeta:=(1,0,0)$, and for each $R>1$ define
\begin{equation*}
\widehat\sigma_{R}(x):=\sum_{g\in G_m}\phi(g)\,\omega(x-Rg\zeta),\qquad x\in\rn.
\end{equation*}
Note that $\widehat\sigma_{R}(gx)=\phi(g)\widehat\sigma_{R}(x)$ for every $g\in G_m$, $x\in\rn$. As in \emph{\cite[Lemmas 4.6 and 4.7]{cmp}} one shows that there exists $R_0>0$ and for each $R\geq R_0$ a unique \ $t_{R}>0$ \ such that
\begin{equation*}
\sigma_{R}:=t_{R}\widehat\sigma_{R}\in\cN^{\phi}.
\end{equation*}
Furthermore, $t_R\to 1$ as $R\to\infty$.
\end{example}

Our next goal is to prove the following result.

\begin{proposition}
\label{prop:upperbound}
If 
\begin{equation} \label{eq:m}
m\geq\sqrt{2}\pi\Big(\frac{\pi}{\sqrt{2}}\Big)^\frac{1}{N-3}
\end{equation}
then, for $R$ large enough,
$$c^{\phi} \leq J(\sigma_{R}) < 2mc_0,$$
where $c_0$ is the ground state energy of \eqref{prob}.
\end{proposition}

We start with some lemmas. The first one is a refinement of \cite[Lemma 4.1]{cm}.

\begin{lemma} \label{lem:cm}
Let $y_1,\ldots,y_n$ be $n$ different points in $\rn$ and $\vartheta_1,\ldots,\vartheta_n$ be positive numbers with $\vartheta:=\vartheta_1+\cdots+\vartheta_n>N$. Then there exists $C=C(\vartheta_i, N)>0$ such that
\[\irn\prod_{i=1}^n(1+|x-Ry_i|)^{-\vartheta_i}\d x\leq Cd^{-\mu} R^{-\mu}\]
for all $R\geq1$, where $d:=\min\{|y_i-y_j|:i,j=1,\ldots,n, \ i\neq j\}$ and  $\mu:=\min\{\vartheta-\vartheta_i,\,\vartheta-N:i=1,\ldots,n\}$.
\end{lemma}

\begin{proof}
Set $\rho:=\frac{1}{2}d>0$ and
\[H_i:=\{z\in\rn:|z-Ry_j|\geq|z-Ry_i|\text{ for every }j\neq i\}.\]
Note that $B_{\rho R}(Ry_i)\subset H_i$ and $H_1\cup\cdots\cup H_n=\rn$. 

Henceforth $C$ will denote different positive constants depending on $\vartheta_i$ and $N$. If $|x-Ry_j|\leq\rho R,$ then $|x-Ry_i|\geq\rho R$ for every $i\neq j$. Thus, for each $j=1,\ldots,n$ we have
\begin{align*}
&\int_{B_{\rho R}(Ry_j)}\prod_{i=1}^n(1+|x-Ry_i|)^{-\vartheta_i}\d x
\leq \int_{B_{\rho R}(Ry_j)}\frac{\mathrm{d}x}{(1+|x-Ry_j|)^{\vartheta_j}(\rho R)^{\vartheta-\vartheta_j}}
\\
&\qquad=(\rho R)^{-(\vartheta-\vartheta_j)}\int_{B_{\rho R}(0)}\frac{\mathrm{d}x}{(1+|x|)^{\vartheta_j}}
\leq C\left[(\rho R)^{-(\vartheta-\vartheta_j)}+(\rho R)^{N-\vartheta}\right]\leq C(\rho R)^{-\mu}.
\end{align*}
Setting $x=Rz$ we obtain
\begin{align*}
&\int_{H_j\smallsetminus B_{\rho R}(Ry_j)}\prod_{i=1}^n(1+|x-Ry_i|)^{-\vartheta_i}\d x \leq \int_{H_j\smallsetminus B_{\rho R}(Ry_j)}\frac{\d x}{(1+|x-Ry_j|)^{\vartheta}} \\
&\qquad  \leq\int_{H_j\smallsetminus B_{\rho}(y_j)}\frac{R^N\d z}{(R|z-y_j|)^\vartheta}\leq CR^{N-\vartheta}\int_{\rho}^{+\infty}\frac{r^{N-1}}{r^\vartheta}\d r \leq C(\rho R)^{-\mu}.
\end{align*}
This completes the proof.
\end{proof}

\begin{lemma} \label{lem:decay}
There are positive constants $C_0$ and $\widehat{C}_0$ such that
\begin{align*}
\lim_{|y| \to \infty} |y|^{N-2} \int_{\mathbb{R}^N}f(\omega(x))\omega(x-y)\d x &=C_0 \\
\lim_{|y| \to \infty} |y|^{N-2} \int_{\mathbb{R}^N}|\omega(x)|^{2^*-1}\omega(x-y)\d x &=\widehat{C}_0.
\end{align*}
\end{lemma}

\begin{proof}
This is proved in Appendix A.
\end{proof}

\begin{lemma}\label{lem:cmp}
There exists $C_1>0$ such that
$$\Big|\irn(tf(u)-f(tu))v\Big|\leq C_1|t-1|\irn|u|^{2^*-1}|v|\qquad\text{for all }t\in[0,2], \ u,v\in D^{1,2}(\rn).$$
\end{lemma}

\begin{proof}
Fix $u\in\r$ and let $h(t):=tf(u)-f(tu)$. It follows from $(f_1)$ that $|h'(t)|\leq C_1|u|^{2^*-1}$ for all $t\in[0,2]$. Hence, $|tf(u)-f(tu)|=|h(t)-h(1)|\leq C_1|u|^{2^*-1}|t-1|$, which implies the conclusion.
\end{proof}

\begin{lemma}\label{lem:acp}
Given $m\in\n$ and $\bar u>0$, there exists $C_2>0$ such that
$$
\left|F\Big(\sum_{i=1}^{2m}u_i\Big)-\sum_{i=1}^{2m}F(u_i)-\sum_{\substack{i,j=1 \\ i\neq j}}^{2m}f(u_i)u_j\right|\leq C_2\left(\sum_{\substack{i,j=1 \\ i<j}}^{2m}|u_iu_j|^{1+\frac\alpha2} + \sum_{\substack{i,j,k=1 \\ i<j<k}}^{2m}|u_iu_j|^\alpha|u_k| \right),
$$ 
for any $u_1,\ldots,u_{2m}\in[-\bar u,\bar u]$.
\end{lemma}

\begin{proof}
This is proved in Appendix B.
\end{proof}

\begin{proof}[Proof of Proposition~\ref{prop:upperbound}]
We write the $G_m$-orbit of $\zeta=(1,0,0)$ as
\[G_m\zeta=\left\{\zeta_{1},\dots,\zeta_{2m}\right\}\quad\text{with}\quad\zeta_i:=\mathrm{e}^{2\pi\mathrm{i}(i-1)/m}\zeta\quad\text{and}\quad\zeta_{m+i}:=\tau\zeta_i,\quad i=1,\ldots,m,\]
and set
\begin{equation*}
\omega_{iR}(x):=
\begin{cases}
\omega(x-R\zeta_i) &\text{for \ }i=1,\ldots,m,\\
-\omega(x-R\zeta_i) &\text{for \ }i=m+1,\ldots,2m.
\end{cases}
\end{equation*}
Then, $\sigma_R=\sum_{i=1}^{2m}t_R\omega_{iR}$; see Example \ref{example}. Note that $J(\omega_{iR})=J(\omega)$. As $\omega_{iR}$ solves \eqref{prob}, from Lemmas \ref{lem:cmp} and \ref{lem:decay} we derive
\begin{align*}
\langle t_R\omega_{iR},t_R\omega_{jR}\rangle &=t_R^2\irn f(\omega_{iR})\omega_{jR}\leq \irn f(t_R\omega_{iR})t_R\omega_{jR}+C|t_R-1|\irn |\omega_{iR}|^{2^*-1}|\omega_{jR}| \\
&\leq \irn f(t_R\omega_{iR})t_R\omega_{jR}+C|t_R-1||\zeta_i-\zeta_j|^{2-N}R^{2-N},
\end{align*}
for $R$ large enough. Using this inequality and Lemmas \ref{lem:acp} and \ref{lem:cm} we obtain
\begin{align} \label{eq:Jsigma}
J(\sigma_{R})=&\frac{1}{2}\Big\|\sum_{i=1}^{2m}t_R\omega_{iR} \Big\|^2-\irn F\Big(\sum_{i=1}^{2m}t_R\omega_{iR}\Big) \nonumber\\
=&\sum_{i=1}^{2m}\frac{1}{2}\left\|t_R\omega_{iR} \right\|^2 - \sum_{i=1}^{2m}\irn F(t_R\omega_{iR}) + \frac{1}{2}\sum_{\substack{i,j=1 \\ i\neq j}}^{2m}\langle t_R\omega_{iR},t_R\omega_{jR}\rangle \nonumber\\
& - \irn F\left(\sum_{i=1}^{2m}t_R\omega_{iR}\right) + \sum_{i=1}^{2m}\irn F(t_R\omega_{iR}) \nonumber\\
\leq & \sum_{i=1}^{2m}J(t_R\omega_{iR}) - \frac{t_R^2}{2}\sum_{\substack{i,j=1 \\ i\neq j}}^{2m}\irn f(\omega_{iR})\omega_{jR} + C|t_R-1|\sum_{\substack{i,j=1 \\ i\neq j}}^{2m}|\zeta_i-\zeta_j|^{2-N}R^{2-N} \nonumber\\
& - \irn\Big(F\Big(\sum_{i=1}^{2m}t_R\omega_{iR}\Big) - \sum_{i=1}^{2m} F(t_R\omega_{iR}) - \sum_{\substack{i,j=1 \\ i\neq j}}^{2m}f(t_R\omega_{iR})t_R\omega_{jR}\Big) \nonumber\\
\leq & 2mJ(\omega) - \frac{t_R^2}{2}\sum_{\substack{i,j=1 \\ i\neq j}}^{2m}\irn f(\omega_{iR})\omega_{jR} + C|t_R-1|R^{2-N} \nonumber\\
& + C\sum_{\substack{i,j=1 \\ i<j}}^{2m}\irn|\omega_{iR}\omega_{jR}|^{1+\frac{\alpha}{2}} +  C\sum_{\substack{i,j,k=1 \\ i<j<k}}^{2m}\irn|\omega_{iR}\omega_{jR}|^\alpha|\omega_{kR}| \nonumber\\
\leq & 2mc_0 - \frac{t_R^2}{2}\sum_{\substack{i,j=1 \\ i\neq j}}^{2m}\irn f(\omega_{iR})\omega_{jR} + C|t_R-1|R^{2-N} + CR^{-\mu_1} + CR^{-\mu_2},
\end{align}
where $C$ stands for different positive constants that do not depend on $R$, $\mu_1:=\min\{(1+\frac{\alpha}{2})(N-2),\,(2+\alpha)(N-2)-N\}$ and $\mu_2:=\min\{(\alpha+1)(N-2),\, 2\alpha(N-2),\,(2\alpha+1)(N-2)-N\}$. Since, by assumption $(f_1)$, $\alpha>\frac{N}{2(N-2)}$, we have that $\mu_1,\mu_2>N-2$.

Next, we estimate the sign of the second summand in the last row. To this end, set \ $d_{ij}:=|\zeta_{i}-\zeta_{j}|$ for $i\neq j$. \ Note that $d_{12}=2\sin(\frac{\pi}{m})$, and $d_{ij}=\sqrt{2}$ if $1\leq i\leq m<j\leq 2m$. Therefore,
\begin{align*}
&\sum_{\substack{i,j=1 \\ i\neq j}}^md_{ij}^{2-N}-\sum_{i=1}^m\sum_{j=m+1}^{2m}d_{ij}^{2-N} \geq \sum_{i=1}^m(d_{i\,(i+1)}^{2-N}+d_{(i-1)\,i}^{2-N}) - \sum_{i=1}^m\sum_{j=m+1}^{2m}d_{ij}^{2-N} \\
&=2m\Big(2\sin\Big(\frac{\pi}{m}\Big)\Big)^{2-N}-m^2(\sqrt{2}\,)^{2-N} > m\Big[2\Big(\frac{2\pi}{m}\Big)^{2-N}-m(\sqrt{2}\,)^{2-N}\Big]\geq 0
\end{align*}
because, by assumption,
\begin{equation*}
m\geq \sqrt{2}\pi\Big(\frac{\pi}{\sqrt{2}}\Big)^\frac{1}{N-3}.
\end{equation*}
Let $C_0$ be as in Lemma \ref{lem:decay} and fix $\eps\in(0,C_0)$ such that
\begin{align*}
M_0:=& 2(C_0-\eps)\sum_{\substack{i,j=1 \\ i\neq j}}^md_{ij}^{2-N} - 2(C_0+\eps)\sum_{i=1}^m\sum_{j=m+1}^{2m}d_{ij}^{2-N}>0.
\end{align*}
Then, for $R$ large enough we have that
\begin{align*}
\sum_{\substack{i,j=1 \\ i\neq j}}^{2m}\irn f(\omega_{iR})\omega_{jR}&=2\sum_{\substack{i,j=1 \\ i\neq j}}^m\irn f(\omega)\omega(\,\cdot\,-R(\zeta_j-\zeta_i)) \\
&\quad-2\sum_{i=1}^m\sum_{j=m+1}^{2m}\irn f(\omega)\omega(\,\cdot\,-R(\zeta_j-\zeta_i))\geq M_0R^{2-N},
\end{align*}
and we derive from \eqref{eq:Jsigma} that
$$J(\sigma_R)\leq 2mc_0 - \frac{t_R^2}{2}M_0R^{2-N} + C|t_R-1|R^{2-N} + o(R^{2-N}).$$
Since $M_0>0$ and $t_R\to 1$ as $R\to\infty$, we conclude that $J(\sigma_{R}) < 2mc_0$ for $R$ large enough, as claimed.
\end{proof}

\begin{remark} \label{rem:m}
The function $\psi(t):=\sqrt{2}\pi\left(\frac{\pi}{\sqrt{2}}\right)^\frac{1}{t}$ is decreasing in $t>0$. Since $\psi(t)\to\sqrt{2}\pi$ as $t\to\infty$ and $\sqrt{2}\pi>4$, any number $m$ satisfying \eqref{eq:m} must be greater than or equal to $5$. Direct computation shows that the least integer greater than or equal to $\sqrt{2}\pi\left(\frac{\pi}{\sqrt{2}}\right)^\frac{1}{N-3}$ is $6$ if $N=5,6$, and it is $5$ if $N\geq 7$.
\end{remark}
\medskip

\begin{proof}[Proof of Theorem \ref{thm:main}]
Let $m$ satisfy \eqref{eq:m}. Take $G:=G_m$ and $\phi$ as in Example \ref{example}. For $\xi=(z_1,z_2,y)\in\cc\times\cc\times\r^{N-4}$ we have that $|G\xi|=2m$ and $G_\xi=\{1\}$ if $z_1\neq z_2$, $|G\xi|=m$ and $G_\xi=\{1,\tau\}$ if $z_1=z_2\neq 0$, and $G_\xi=G$ if $z_1=z_2=0$. So, according to Proposition \ref{prop:compactness}, $c^\phi$ is attained if $c^\phi<\min\{2mc_0,mc^{\phi|\{1,\tau\}}\}$. 

If $u\in\cN^{\phi|\{1,\tau\}}$ then $u$ changes sign and $u^-(x)=-u^+(\tau x)$ where $u^+:=\max\{u,0\}$ and $u^-:=\min\{u,0\}$. Therefore,
$$\|u\|^2=2\|u^\pm\|^2,\qquad\irn F(u)=2\irn F(u^\pm),\qquad\irn f(u)u=2\irn f(u^\pm)u^\pm,$$
and, as a consequence, 
$$u^\pm\in\cN:=\{v\in D^{1,2}(\rn):v\neq 0\text{ \ and \ }\|v\|^2=\irn F(v)\}$$ 
and $J(u)=2J(u^\pm)\geq 2c_0$. It follows that \ $\min\{2mc_0,mc^{\phi|\{1,\tau\}}\}=2mc_0$.
 
Proposition \ref{prop:upperbound} asserts that $2mc_0>c^\phi$. Therefore $c^\phi$ is attained at some $\overline{\omega}\in\cN^\phi$. Furthermore, as $\cN^{\phi}\subset\cN^{\phi|\{1,\tau\}}$, we have that $c^{\phi|\{1,\tau\}}\leq c^\phi$. If $J(\overline{\omega})=2c_0$, then $J(\overline{\omega}^+)=c_0$, contradicting the maximum principle. This shows that $c^\phi>2c_0$.

As $\overline{\omega}\in\cN^\phi$, it has properties $(a)$ and $(b)$. Property $(c)$ holds true after a suitable translation. Indeed, if $a_i\in\r$ and $\varrho_i:\rn\to\rn$ is the reflection on the hyperplane $\{(z_1,z_2,y_1,\ldots,y_{N-4}):y_i=a_i\}$, the function
\begin{equation*}
w(z_1,z_2,y):=
\begin{cases}
\overline{\omega}(z_1,z_2,y) &\text{if \ }y_i>a_i,\\
(\overline{\omega}\circ\varrho_i)(z_1,z_2,y) &\text{if \ }y_i<a_i,
\end{cases}
\end{equation*}
is also in $\cN^\phi$ and $J(w)=c^\phi$. Applying Lopes' method \cite{l} one shows that there exists $a=(a_1,\ldots,a_{N-4})$ such that $\overline{\omega}$ is invariant under the reflection on every hyperplane $y_i=a_i$. It follows that $\widehat{\omega}(z_1,z_2,y):=\overline{\omega}(z_1,z_2,y+a)$ belongs to $\cN^\phi$ and satisfies $(c)$ and $J(u)=c^\phi$.

Remark \ref{rem:m} completes the proof.
\end{proof}

\appendix 

\section{ The proof of Lemma \ref{lem:decay}}

We start by establishing the exact decay of the ground state.

\begin{lemma} \label{lem:decay ground state}
Let $\omega\in\cC^2(\rn)$ be a positive radial ground state solution to \eqref{prob}. Then there exists $c_0>0$ such that 
$$\lim_{t\to\infty}t^{N-2}\omega(t)=c_0.$$
\end{lemma}
\begin{proof}
By \eqref{eq:decay} there exist two positive constants $c_{1},\, c_{2}$ such that
\[
0<c_{1}\leq r^{N-2}\omega(r)\leq c_{2} \qquad \text{for all \ } r\geq 1.
\]
So it suffices to show that there exists $\rho>0$ such that the function $r\mapsto r^{N-2}\omega(r)$ is monotone in the interval $(\rho,\infty)$ or, equivalently, that the function $v: [1,\infty)\to (0,\infty)$ defined as $v(r):=\frac1{r^{N-2}\omega(r)}$ is monotone in $(\rho,\infty)$. To this aim, observe that 
\[
\Delta(r^{-(N-2)})= (r^{-(N-2)})''+\frac{N-1}r (r^{-(N-2})'=0,
\]
so that
\begin{align*}
0&=\Delta (v\omega)=\omega v''+2v'\omega'+v\omega''+\frac{N-1}r\left(v'\omega+v\omega'\right)
\\
&=\omega\left[v''+\left(2\frac{\omega'}\omega+\frac{N-1}r\right)v'\right]+\left(\omega''
+\frac{N-1}r\omega'\right)v.
\end{align*}
Then, $v$ satisfies 
\[
v''+b(r)v' +c(r)v=0, \qquad \text{with \ \ }b(r):=\frac{2\omega'}\omega+\frac{N-1}r\text{ \ and \ }c(r):=-\frac{f(\omega)}\omega<0.
\]
If there exists $\rho>0$ such that $v'(r)\geq 0$ for every $r\geq\rho$, or $v'(r)\leq 0$ for every $r\geq\rho$, then $v$ is monotone in $(\rho,\infty)$.
On the other hand, if such $\rho$ does not exist, then $v'$ changes sign infinitely many times in $(1,\infty)$. In particular, $v$ has a local maximum $r_0$ in $[1,\infty)$. But this implies that 
\[
0\geq v''(r_{0})=-c(r_{0})v(r_{0})>0,
\]
a contradiction. The proof is complete.
\end{proof}

\begin{remark}
In fact, $v=\frac1{r^{N-2}\omega}$ is nonincreasing. Indeed, defining $\underline{u}:=\frac{c_{1}}{r^{N-2}}$ and using \eqref{eq:decay}, for every $\rho>0$ one has
\[
\begin{cases}
\Delta(\underline{u}-\omega) =f(\omega) & \text{in \ }[\rho,+\infty),
\\
\left(\underline{u}-\omega\right)(\rho)\leq 0.
\end{cases}
\]
Then, the maximum principle implies that $\omega>\underline{u}$ and Hopf's Lemma yields $(\underline{u}-\omega)'(\rho)<0$.
Therefore $v=\frac{\underline{u}}{c_{1}\omega}$ satisfies
\[
v'(\rho)=\frac1{c_{1}}\frac{\underline{u}'\omega-\omega'\underline{u}}{\omega^{2}}
< \frac1{c_{1}}\frac{\omega'\omega-\omega'\underline{u}}{\omega^{2}}
=\frac{\omega'}{c_{1}}\frac{\omega-\underline{u}}{\omega^{2}}\leq 0,
\]
as  $\omega$ is decreasing.
\end{remark}

\begin{proof}[Proof of Lemma \ref{lem:decay}]
From Lemma \ref{lem:decay ground state} for each fixed $x\in\rn$ we get that
\begin{align*}
\lim_{|y|\to\infty}|y|^{N-2}f(\omega(x))\omega(x-y)&=\lim_{|y|\to\infty}\left|\frac{y}{x-y}\right|^{N-2}f(\omega(x))|x-y|^{N-2}\omega(x-y)\\
&=c_0f(\omega(x)).
\end{align*}
Note that, as $\omega$ is decreasing in the radial direction,
\begin{equation*}
|y|^{N-2}\omega(x-y)\leq
\begin{cases}
|x-y|^{N-2}\omega(x-y) &\text{if \ }|y|\leq |y-x|,\\
|y|^{N-2}\omega(y) &\text{if \ }|y|\geq |y-x|.
\end{cases}
\end{equation*}
So from \eqref{eq:decay} we get that $|y|^{N-2}f(\omega(x))\omega(x-y)\leq Cf(\omega(x))$ for every $y\in\rn$. Assumption $(f_1)$ and \eqref{eq:decay} yield $f(\omega(x))\leq C|x|^{-(q-1)(N-2)}$ and $q-1>\frac{N+2}{N-2}$. Therefore $f\circ\omega$ is integrable in $\rn$ and by the dominated convergence theorem
$$\lim_{|y|\to\infty}\irn |y|^{N-2}f(\omega(x))\omega(x-y)\d x = c_1\irn f(\omega(x))\d x.$$
The other identity is obtained in a similar way.
\end{proof}

\section{The proof of Lemma \ref{lem:acp}}

\begin{lemma}\label{lem:f}
Given $n\in\n$, $\bar u>0$ and $f\in \cC_\mathrm{loc}^{1,\beta}(\R)$ with $\beta\in(0,1]$ such that $f(0)=0$, there exists $b_1>0$ such that
$$\left|f\Big(\sum_{i=1}^nu_i\Big)-\Big(\sum_{i=1}^nf(u_i)\Big)\right|\leq b_1\sum_{\substack{i,j=1 \\ i<j}}^n|u_iu_j|^\beta\qquad\text{for any}\quad u_1,\ldots,u_n\in[-\bar u,\bar u].$$
\end{lemma}

\begin{proof}
We argue by induction on $n$. As $f\in \cC_\mathrm{loc}^{1,\beta}(\R)$ and $f(0)=0$, we have that
\begin{align*}
&|f(u+v)-f(u)-f(v)|=\Big|\int_0^u(f'(s+v)-f'(s))\d s\Big|\\
&\leq\Big|\int_0^u|f'(s+v)-f'(s)|\d s\Big|\leq C|v|^\beta|u|\leq C\,\bar u^{1-\beta}|uv|^\beta\quad\text{for any \ }u,v\in[-\bar u,\bar u].
\end{align*}
Assume the result is true for $n-1\geq 2$ and let $u_1,\ldots,u_n\in[-\bar u,\bar u]$. Then,
\begin{align*}
&\Big|f\Big(\sum_{i=1}^nu_i\Big)-\Big(\sum_{i=1}^nf(u_i)\Big)\Big|\\
&\qquad\leq \Big|f\Big(\sum_{i=1}^nu_i\Big)-f\Big(\sum_{i=1}^{n-1}u_i\Big)-f(u_n)\Big|+\Big|f\Big(\sum_{i=1}^{n-1}u_i\Big)-\Big(\sum_{i=1}^{n-1}f(u_i)\Big)\Big| \\
&\qquad \leq C\Big(\Big|\sum_{i=1}^{n-1}u_iu_n\Big|^\beta + \sum_{\substack{i,j=1 \\ i<j}}^{n-1}|u_iu_j|^\beta \Big)\leq b_1\sum_{\substack{i,j=1 \\ i<j}}^n|u_iu_j|^\beta,
\end{align*}
as claimed.
\end{proof}

\begin{lemma}\label{lem:acp2}
Given $n\in\n$, $\bar u>0$ and $f\in \cC_\mathrm{loc}^{1,\beta}(\R)$ with $\beta\in(0,1]$ such that $f(0)=0=f'(0)$, there exists $b_2>0$ such that
$$
\left|F\Big(\sum_{i=1}^nu_i\Big)-\sum_{i=1}^nF(u_i)-\sum_{\substack{i,j=1 \\ i\neq j}}^nf(u_i)u_j\right|\leq b_2\left(\sum_{\substack{i,j=1 \\ i<j}}^n|u_iu_j|^{1+\frac\beta2} + \sum_{\substack{i,j,k=1 \\ i<j<k}}^n|u_iu_j|^\beta|u_k| \right),
$$ 
for any $u_1,\ldots,u_n\in[-\bar u,\bar u]$, where $F(u):=\int_0^uf$.
\end{lemma}

\begin{proof}
We argue by induction on $n$. Following \cite[Lemma 3]{lions90}, we define 
\[
G(u,v):=F(u+v)-F(u)-F(v)-f(u)v-f(v)u,\qquad  u,v\in[-\bar u,\bar u].
\] 
Then $\partial_u\partial_vG(u,v)=f'(u+v)-f'(u)-f'(v)$ and, as $f\in \cC_\mathrm{loc}^{1,\beta}(\R)$ and $f'(0)=0$,
\begin{align*}
\left|\partial_u\partial_vG(u,v)\right|\leq\left|f'(u+v)-f'(u)\right|+\left|f'(v)\right|\leq C|v|^{\beta},\\
\left|\partial_u\partial_vG(u,v)\right|\leq\left|f'(u+v)-f'(v)\right|+\left|f'(u)\right|\leq C|u|^{\beta}.
\end{align*}
Therefore,
\[
\left|\partial_u\partial_vG(u,v)\right|\leq C\min\{|u|, |v|\}^{\beta}.
\]
Taking into account that $G(0,v)=\partial_uG(s,0)=0$ we get
\begin{align*}
\left|G(u,v)\right|&=
\left|\int_{0}^u\partial_uG(s,b)\d s\right|=\left|\int_{0}^u\int_{0}^v\partial_v\partial_uG(s,t)\d t\d s\right|\leq \left|\int_{0}^u\int_{0}^v| \partial_v\partial_uG(s,t)|\d t\d s\right|\\
&\leq C\left|\int_{0}^u\int_{0}^v\min\{|s|,|t|\}^{\beta}\d t\d s\right|\leq C\min\{|u|,|v|\}^{\beta}|uv|\leq b_2|uv|^{1+\frac{\beta}{2}},
\end{align*}
because
\[
\min\{|u|,|v|\}^{\beta}|uv|=|u|^{1+\frac\beta2+\frac\beta2}|v|\leq |uv|^{1+\frac\beta2}\qquad\text{for \ }|u|\leq |v|.
\]
This proves the statement for $n=2$. Assume the inequality holds true for $n-1\geq 2$ and let $u_1,\ldots,u_n\in[-\bar u,\bar u]$. Then, using
Lemma \ref{lem:f} we obtain
\begin{align*}
&
\Big|F\Big(\sum_{i=1}^nu_i\Big)-\sum_{i=1}^nF(u_i)-\sum_{\substack{i,j=1 \\ i\neq j}}^nf(u_i)u_j\Big| 
\\
&\leq\Big|F\Big(\sum_{i=1}^nu_i\Big)-F\Big(\sum_{i=1}^{n-1}u_i\Big)-F(u_n)-f\Big(\sum_{i=1}^{n-1}u_i\Big)u_n-\sum_{i=1}^{n-1}u_if(u_n)\Big| 
\\
&\qquad + \Big|F\Big(\sum_{i=1}^{n-1}u_i\Big)-\sum_{i=1}^{n-1}F(u_i)-\sum_{\substack{i,j=1 \\ i\neq j}}^{n-1}f(u_i)u_j\Big| + \Big|f\Big(\sum_{i=1}^{n-1}u_i\Big)u_n-\sum_{i=1}^{n-1}f(u_i)u_n\Big|
\\
&\leq C\Big(\Big|\sum_{i=1}^{n-1}u_iu_n\Big|^{1+\frac\beta2} + \sum_{\substack{i,j=1 \\ i<j}}^{n-1}|u_iu_j|^{1+\frac\beta2} + \sum_{\substack{i,j,k=1 \\ i<j<k}}^{n-1}|u_iu_j|^\beta|u_k| + \sum_{\substack{i,j=1 \\ i<j}}^{n-1}|u_iu_j|^\beta|u_n|\Big) 
\\
&\leq C\Big(\sum_{\substack{i,j=1 \\ i<j}}^{n}|u_iu_j|^{1+\frac\beta2} + \sum_{\substack{i,j,k=1 \\ i<j<k}}^{n}|u_iu_j|^\beta|u_k| \Big),
\end{align*}
as claimed.
\end{proof}

\begin{remark}
Note that the growth condition on $f'$ in $(f_1)$ does not imply that $f\in \cC^{1,\alpha}$ even in a small interval $[0,\delta]$.
\end{remark}

 \vspace{15pt}

\begin{flushleft}
\textbf{Mónica Clapp}\\
Instituto de Matemáticas\\
Universidad Nacional Autónoma de México \\
Campus Juriquilla\\
Boulevard Juriquilla 3001\\
76230 Querétaro, Qro.\\
Mexico\\
\texttt{monica.clapp@im.unam.mx} \vspace{10pt}

\textbf{Liliane A. Maia}\\
Departamento de Matemática\\
Universidade de Brasília \\
70910-900 Brasília\\
Brazil\\
\texttt{lilimaia@unb.br} \vspace{10pt}

\textbf{Benedetta Pellacci}\\
Dipartimento di Matematica e Fisica\\
Universit\`a della Campania ``Luigi Vanvitelli''\\
Viale Lincoln 5\\
81100 Caserta\\
Italy\\
\texttt{benedetta.pellacci@unicampania.it}
\end{flushleft}

\end{document}